\documentclass{elsarticle}

\usepackage{hyperref, natbib}
\usepackage{graphicx, amssymb, amsmath, amsthm, epsfig, epstopdf}
\usepackage[title]{appendix}
\journal{Discrete Applied Mathematics}
\theoremstyle{plain}
\newtheorem{theorem}{Theorem}[section]
\newtheorem{corollary}{Corollary}[section]
\newtheorem{lemma}{Lemma}[section]

\theoremstyle{definition}
\newtheorem{definition}{Definition}[section]
\newtheorem{remark}{Remark}[section]
\newtheorem*{coloringproblem}{The a-graph coloring problem}
\newtheorem*{systematicsearch}{Systematic search}
\newtheorem*{constructingconfigurations}{Constructing the conforming configurations}
\newtheorem*{coloringconfigurations}{Coloring the conforming configurations}
\newtheorem*{results}{Results}

\numberwithin{equation}{section}
\numberwithin{figure}{section}
\numberwithin{table}{section}

\begin{document}

\begin{frontmatter}

\title{The a-graph coloring problem}

\author[affiliation]{J.A. Tilley}
\address[affiliation]{Retired, unaffiliated, 61 Meeting House Road, Bedford Corners, NY 10549, U.S.A.}
\ead{jimtilley@optonline.net}


\begin{abstract}
\noindent No proof of the 4-color conjecture reveals why it is true; the goal has not been to go beyond proving the conjecture. The standard approach involves constructing an unavoidable finite set of reducible configurations to demonstrate that a minimal counterexample cannot exist. We study the 4-color problem from a different perspective. Instead of planar triangulations, we consider near-triangulations of the plane with a face of size 4; we call any such graph an a-graph. We state an a-graph coloring problem equivalent to the 4-color problem and then derive a coloring condition that a minimal a-graph counterexample must satisfy, expressing it in terms of equivalence classes under Kempe exchanges. Through a systematic search, we discover a family of a-graphs that satisfy the coloring condition, the fundamental member of which has order 12 and includes the Birkhoff diamond as a subgraph. Higher-order members include a string of Birkhoff diamonds. However, no member has an applicable parent triangulation that is internally 6-connected, a requirement for a minimal counterexample. Our research suggests strongly that the coloring and connectivity conditions for a minimal counterexample are incompatible; infinitely many a-graphs meet one condition or the other, but we find none that meets both. 
\end{abstract}


\end{frontmatter}


\section{Graph terminology}

\noindent We use standard graph terminology. All graphs considered in this article are assumed to be planar unless otherwise stated. $V(G)$ denotes the \textit{vertex set} of the graph $G$. Two vertices are \textit{adjacent} if they share an edge. Such an edge is said to be \textit{incident} to the two vertices it joins. A \textit{triangulation} is a graph in which all faces are delineated by three edges and a \textit{near-triangulation} is a graph in which all faces but one are delineated by three edges. In this article, we refer to the edges delineating the sole non-triangular face in a near-triangulation as the \textit{boundary} of the graph and the vertices on the boundary as \textit{boundary vertices}. Vertices not on the boundary are referred to as \textit{interior vertices}. An \textit{internal path} is one having no edge on the boundary. A \textit{separating $n$-cycle} in $G$ has vertices of $G$ both inside and outside the $n$-cycle. A connected graph is \textit{$k$-connected} if it has more than $k$ vertices and remains connected whenever fewer than $k$ vertices are deleted. In an \textit{$r$-regular} graph, all vertices have degree $r$. A \textit{proper vertex-coloring} of $G$ is a coloring of $V(G)$ in which no two adjacent vertices have the same color. The only graph colorings we consider are proper vertex-colorings; we often refer to them merely as \textit{colorings}. We use the numbers 1, 2, 3, 4, ... to denote the different colors available for coloring a graph. We often use the term 4-\textit{coloring} to mean any coloring using no more than 4 colors. The expression $c(w)$ translates as \textit{the color of vertex $w$ in the coloring $c$}. A \textit{Kempe chain} is a maximal connected subgraph of $G$ whose vertices in a coloring of $G$ use only two colors. A Kempe chain that uses the colors $j$ and $k$ is referred to as a $j$-$k$ \textit{chain}. A vertex $w$ colored $j$ that is not adjacent to any vertex colored $k \neq j$ is a \textit{single-vertex} or \textit{short} $j$-$k$ \textit{Kempe chain}. Exchanging colors $j$ and $k$ for such a chain is the same as changing the color of $w$ to $k$.

For purposes of this article, a \textit{Kempe exchange} is a recoloring of $G$ in which the colors $j$ and $k$ are exchanged in a designated, non-empty, proper subset of all $j$-$k$ Kempe chains. For example, if $G$ has three pairwise disjoint 1-2 Kempe chains, then exchanging colors on any one of those chains or any two of those chains counts as a single Kempe exchange.


\section{Introduction}

\noindent This article is primarily about discovering why the 4-color conjecture is true. No existing proof \citep{AppelHaken, AppelHakenKoch, RobertsonSandersSeymourThomas, Steinberger} accomplishes that{---}it was never the goal. Succeeding at the endeavor may open up a new avenue by which the 4-color conjecture can ultimately be proved without relying on a computer. Even if that turns out not to be the case, we contend that it is worthwhile trying to understand what it is that renders the 4-color problem solvable.

In its tightest formulation, the statement of the 4-color problem is to show that any planar triangulation has a 4-coloring.  Instead of planar triangulations, we study near-triangulations of the plane in which the sole non-triangular face has size 4. We call any such near-triangulation an \textit{a-graph} because it is an ``almost-triangulated-graph.'' Instead of the 4-color problem, we study the a-graph coloring problem.

\begin{coloringproblem} Let $xy$ be any edge in an arbitrary planar triangulation $T$. Show that the a-graph $G = T - xy$ has a 4-coloring $c$ in which $c(x) \neq c(y)$.
\end{coloringproblem}

\noindent The 4-color problem and the a-graph coloring problem are trivially equivalent. Start with an uncolored $T$ and delete the edge $xy$, give the resulting $G$ a coloring $c$ that solves the a-graph coloring problem, then replace the edge $xy$ to obtain a 4-coloring of $T$. Conversely, start with an uncolored $G$, insert the edge $xy$, give the resulting $T$ a coloring $c$ that solves the 4-color problem, then delete the edge $xy$ to obtain the required 4-coloring of $G$.

What can possibly be achieved by this near sleight-of-hand? What is the motivation for studying the a-graph coloring problem? The answers lie in examining all distinct colorings of both $T$ and $G$. Obviously, all colorings of $T$ are colorings of $G$. But there are colorings $c$ of $G$ in which $c(x) = c(y)$ that are not colorings of $T$. Generally, many of these will be Kempe-equivalent to colorings of $T$. This observation underpins our approach. \textit{The key concept is to obtain a proper coloring of some target graph, say the triangulation $T$, by using colorings of a related graph, say the a-graph $G$, that are not proper colorings of $T$ but are Kempe-equivalent to proper colorings of $T$.} Suppose that any triangulation of lower order than $T$ has a 4-coloring, either because $T$ is a minimal counterexample or purely as an inductive assumption. By contracting the edge $xy$ in $T$ rather than deleting it, we obtain a triangulation that can be given a 4-coloring $c$. Upon reversing the contraction without restoring the $xy$ edge, we arrive at $G$ with 4-coloring $c$ in which $c(x) = c(y)$  and can then proceed to analyze the a-graph coloring problem, attempting to navigate by means of Kempe exchanges from $c$ to a 4-coloring $c^{\prime}$ in which $c^{\prime}(x) \neq c^{\prime}(y)$.

It is useful to enhance our notation regarding the designation and description of an a-graph $G$. We establish the convention of always drawing $G$ with the 4-face as an exterior face and orienting that 4-face as shown in figure 6.1, labeling the boundary cycle $uxvy$, thus establishing $(x,y)$ and $(u,v)$ as the two pairs of \textit{opposite (non-adjacent) boundary vertices}. Then we can refer to the two \textit{parent triangulations} of any a-graph $G$ as $G + xy$ and $G + uv$. We always denote the left-hand vertex on the 4-face by $x$, the right-hand vertex by $y$, the bottom vertex by $u$, and the top vertex by $v$.

With respect to graph structure, the same $G$ results whether the edge $xy$ is deleted in the parent $G + xy$ or the edge $uv$ in the parent $G + uv$. For some purposes, it will suffice to discuss $G$ without reference to a specific parent triangulation, but for many purposes it will not, and in those situations the \textit{applicable parent} needs to be designated. When that is the case, $G$ is best thought of as two different graphs, one with parent $G + xy$ and one with parent $G + uv$. In the first case, we refer to the \textit{parented a-graph} as $G_{xy}$, and in the second case, we refer to the \textit{parented a-graph} as $G_{uv}$. A helpful reminder as to which parent is applicable is to think of $G_{xy}$ as having a \textit {ghost edge} $xy$ and of $G_{uv}$ as having a \textit{ghost edge} $uv$. 

Clearly the parented a-graphs $G_{xy}$ and $G_{uv}$ can differ as to the connectivity of their respective applicable parents and we shall see that they can also differ as to how their equivalence classes under Kempe exchanges are categorized, both matters crucial to determining whether they can be minimal a-graph counterexamples. So, for example, $G_{xy}$ might be a minimal a-graph counterexample while $G_{uv}$ is not. Thus the need to distinguish between them. When the a-graph coloring problem for $G$ applies to the $(x,y)$ pair, we refer to $G$ as $G_{xy}$, and when the a-graph coloring problem for $G$ applies to the $(u,v)$ pair, we refer to $G$ as $G_{uv}$. When we talk about the structure of an a-graph without reference to a particular pair of opposite boundary vertices, we simply refer to the a-graph as $G$. For instance, we might simply say that $G$ has minimum degree 4.


\section{Overview}

\noindent In line with other work, ours is based on the notion of a minimal counterexample. To be a minimal counterexample to the 4-color conjecture, a planar triangulation must be \textit{internally 6-connected}. An internally 6-connected triangulation has minimum degree 5 and no separating 3-cycles or 4-cycles. Further, if any separating 5-cycle is removed, at least one of the components created has only a single vertex. A lucid description of this property can be found in both \citep{RobertsonSandersSeymourThomas} and  \citep{Steinberger}. There are an infinite number of internally 6-connected planar triangulations, each of which is potentially a minimal counterexample. Without a method to sort through them, to characterize and classify them, each would have to be examined individually, an impossible task. That is why mathematicians directed their efforts in the early 1970s towards constructing a \textit{finite} set of configurations that are both \textit{unavoidable} and \textit{reducible}. (Refer to \citep{Wilson}.) \textit{Our approach is different.} Rather than offer yet another unavoidable set of reducible configurations, we ask: Is there some property we can use to winnow the infinite set down to a set with a distinguishing characteristic that renders it amenable to straightforward analysis? There does appear to be, one that is readily expressed for a-graphs. This \textit{coloring property} is the subject of section 5.

To state the coloring property efficiently, we introduce the concept of a \textit{state transformation graph} to depict the partitioning of the set of all distinct colorings of a given parented a-graph into equivalence classes under Kempe exchanges. We label any class one of three types. We express the coloring property as the absence of one of those types. The set of parented a-graphs exhibiting the coloring property we refer to as the family $A^{*}$. In section 6, we provide a heuristic, but compelling, argument as to why a \textit{fundamental} member of $A^{*}$, one without repetitions of a nontrivial configuration, is expected to have low order. Thus, we can reasonably limit the search for fundamental members of $A^{*}$ to a-graphs of order at most 20. In retrospect, this restriction appears to have been justified: the search turned up only a single fundamental member and it has order 12. We refer to that member as $G^{*}$, and, using our conventions and enhanced notation, more precisely as the parented a-graph $G_{xy}^{*}$ of figure 6.1. So, $G_{xy}^{*} \in A^{*}$, but its applicable parent triangulation is not internally 6-connected. The icosahedron is the applicable parent of $G_{uv}^{*}$. It is internally 6-connected, but $G_{uv}^{*} \notin A^{*}$. A \textit{minimal a-graph counterexample} must have an internally 6-connected parent, and, \textit{with that parent}, belong to $A^{*}$. In a sense, $G^{*}$ comes close, but it is $G_{xy}^{*}$ that belongs to $A^{*}$ and $G_{uv}^{*}$ with the internally 6-connected parent. Our systematic search found no parented a-graph that satisfies the coloring condition and whose applicable parent is internally 6-connected. That is our key result. It points to why a minimal counterexample does not exist{---}the coloring and connectivity conditions are incompatible. To be clear, we have not proved that statement, but our research strongly suggests it. 

In a future article, we will show that the methodology described in this article can be used to determine when Kempe chain entanglements are resolvable and when they are not, and also how to resolve them when they are not. Most significantly, no sequence of Kempe exchanges can resolve entangled Kempe chains to find a coloring that solves the a-graph coloring problem for a-graphs belonging to the family $A^{*}$. In the future article, we implement an algorithm that maps out components of a state transformation graph and use it to demonstrate that an explicit solution to the a-graph coloring problem can always be found by Kempe exchanges for a-graphs derived from the well-known Errera, Heawood, and Kittell triangulations \citep{WeissteinErrera, WeissteinHeawood, WeissteinKittell}.


\section{Connectivity property of a minimal a-graph counterexample}

\noindent In a 1913 article \citep{Birkhoff}, Birkhoff proved that a minimal counterexample to the $4$-color conjecture must be internally 6-connected. (He did not use this particular term, but it has become standard.) We show that this condition essentially applies to the a-graph coloring problem. Suppose there is a parented a-graph $G_{xy}$ such that the color of $x$ is the same as the color of $y$ in every 4-coloring of $G_{xy}$. Then $G_{xy}$ is said to be an \textit{a-graph counterexample}. A \textit{minimal a-graph counterexample} is an a-graph counterexample for which there is no a-graph counterexample of lower order. It is natural to think that whatever connectivity property applies to a minimal counterexample to the 4-color conjecture automatically translates to an equivalent property for a minimal a-graph counterexample. This turns out to be right.

\begin{theorem} Deleting any edge in a minimal counterexample to the $4$-color conjecture results in a minimal a-graph counterexample.
\end{theorem}

\begin{proof} Suppose that a planar triangulation $T$ is a minimal counterexample to the 4-color conjecture. Let $xy$ be an arbitrary edge in $T$. Delete $xy$ to form a parented a-graph $G_{xy}$. Suppose $G_{xy}$ is not an a-graph counterexample. Then there exists a 4-coloring $c$ such that $c(x) \neq c(y)$. When the edge $xy$ is inserted, $c$ becomes a 4-coloring of $T$, contradicting the supposition.

Suppose that $G_{xy}$ is not a \textit{minimal} a-graph counterexample. Then there exists a parented a-graph $G_{x^{\prime}y^{\prime}}^{\prime}$ of lower order than $G_{xy}$ that is an a-graph counterexample with a pair of opposite boundary vertices $(x^{\prime},y^{\prime})$ that in any 4-coloring of $G_{x^{\prime}y^{\prime}}^{\prime}$ must be colored the same. However, the triangulation $T^{\prime}$ formed by inserting an edge $x^{\prime}y^{\prime}$ must be 4-colorable since $T^{\prime}$ has lower order than $T$. By deleting the edge $x^{\prime}y^{\prime}$, we  obtain a 4-coloring of $G_{x^{\prime}y^{\prime}}^{\prime}$ in which $x^{\prime}$ and $y^{\prime}$ are not colored the same, a contradiction.
\end{proof}

\begin{theorem} Suppose $G_{xy}$ is a minimal a-graph counterexample. Then the triangulation $G_{xy} + xy$ is a minimal counterexample to the $4$-color conjecture.
\end{theorem}

\begin{proof} Because $c(x) = c(y)$ in all 4-colorings $c$ of $G_{xy}$, the parent triangulation $T$ formed from $G_{xy}$ by inserting an edge $xy$ is a counterexample to the 4-color conjecture, and it is a minimal counterexample because if it were not then $G_{xy}$ would not be a minimal a-graph counterexample.
\end{proof}

\noindent We say that an a-graph is \textit{internally $6$a-connected} if at least one of its two parent triangulations is internally 6-connected.

\begin{corollary} A minimal a-graph counterexample is internally 6a-connected.
\end{corollary}


\section{Coloring property of a minimal a-graph counterexample}

\noindent In a 1975 doctoral thesis \citep{Stromquist}, a year before Appel and Haken announced their proof of the 4-color conjecture, Stromquist showed that a minimal counterexample must have order 52 or higher. That result translates directly to the a-graph coloring problem. For the a-graph perspective to bring something new to the 4-color problem, we must discover a property of a minimal counterexample that, when combined with internal 6-connectedness, is so highly restrictive that it eliminates the possibility that a minimal counterexample exists. A promising candidate is a particular coloring property applying to a-graphs, one that is most clearly expressed in terms of equivalence classes under Kempe exchanges.

\begin{remark} The use of equivalence classes under Kempe exchanges to study various graph-coloring problems is not new. Others \citep{Fisk, VerngasMeyniel, Meyniel, Mohar} have used them in connection with the $k$-coloring of $k$-regular graphs, the 4-coloring of Eulerian triangulations of the plane, and Hadwiger's conjecture.
\end{remark}
 
We often refer to a particular 4-coloring of a parented a-graph $G$ as a \textit{coloring state} of $G$ or simply a \textit{state} of $G$. Two 4-colorings of $G$ represent the same state if they lead to the same partition of $V(G)$ into color classes. Let $C_G$ denote the set of states of $G$. Kempe exchanges induce an equivalence relation among those states and can thus be used to partition $C_G$ into equivalence classes, each such class a component of the \textit{state transformation graph} associated with $G$. The state transformation graph has the states of $G$ as vertices, two states being adjacent if they are connected by a single Kempe exchange.

Consider a parented a-graph $G_{xy}$. Any state in which $c(x) \neq c(y)$ we call a \textit{solution state} for obvious reasons and therefore any state in which $c(x) = c(y)$ we call a \textit{non-solution state}. Any equivalence class all of whose states are non-solution states we label an $n$ \textit{class}, and one for which all states are solution states we label an $s$ \textit{class}. Any equivalence class for which some states are non-solution states and some are solution states we label an $n$-$s$ \textit{class}. In terms of solving the a-graph coloring problem, any $G_{xy}$ which has an $n$-$s$ class or an $s$ class is one with which we do not have to concern ourselves. Any $G_{xy}$ which has only $n$ classes (or a single $n$ class) is an a-graph counterexample.

\begin{remark} The composition of a particular equivalence class{---}specifically, what coloring states constitute the class{---}does not depend on whether it is $G_{xy}$ or $G_{uv}$ that is being considered, but the labels $n$, $s$, or $n$-$s$ that are attached to the various equivalence classes do.
\end{remark}

\noindent Since the motivation for studying the a-graph coloring problem in lieu of the 4-color problem is to bring non-solution states into the picture to serve as starting points to allow Kempe exchanges to find solution states, it makes sense to distinguish parented a-graphs by whether or not they have at least one $n$-$s$ class. All parented a-graphs which lack an $n$-$s$ class we deem particularly worthy of consideration and we use that characteristic to define the family $A^*$.

\begin{definition} A parented a-graph $G_{xy}$ belongs to the family $A^*$ if and only if the Kempe partition of $C_{G_{xy}}$ has no $n$-$s$ class.
\end{definition}

\begin{theorem} A minimal a-graph counterexample belongs to $A^*$.
\end{theorem}

\begin{proof} Let $G_{xy}$ be a minimal a-graph counterexample. Suppose $G_{xy} \not\in A^*$. Then there is at least one non-solution state of $G_{xy}$ equivalent to a solution state under Kempe exchanges, contradicting the supposition.
\end{proof}

\begin{lemma} Let $c$ be a non-solution state of $G_{xy} \in A^*$. Then there exist $2$-color paths between $x$ and $y$ for all three color-pairings that include the common color of $x$ and $y$.
\end{lemma}

\begin{proof} Without loss of generality, we have $c(x) = c(y) = 1$. Suppose there is no 1-$k$ path between $x$ and $y$ for some $k \neq 1$. Then there is a 1-$k$ Kempe chain (perhaps a short chain) that includes $y$ but not $x$. Exchanging colors on that chain yields a solution state and indicates the presence of an $n$-$s$ class, contradicting the statement that $G_{xy} \in A^*$. Thus, there is a 1-$k$ path between $x$ and $y$ for $k = 2$, 3, and 4.
\end{proof}

\begin{lemma} Let $c$ be a solution state of $G_{xy} \in A^*$. Then there exists a $2$-color path between $x$ and $y$.
\end{lemma}

\begin{proof} Without loss of generality, we have $c(x) = 1$ and $c(y) = 2$. Suppose there is not a 1-2 path between $x$ and $y$. Then there is a 1-2 Kempe chain (perhaps a short chain) that includes $y$ but not $x$. Exchanging colors on that chain yields a non-solution state for $G_{xy}$ and indicates the presence of an $n$-$s$ class, contradicting the assumption that $G_{xy} \in A^*$. Thus, there is a 1-2 path between $x$ and $y$.
\end{proof}

\begin{lemma} $G_{xy} \in A^*$ if and only if in every non-solution state there is no $2$-color path between $u$ and $v$ using colors different from the common color of $x$ and $y$.
\end{lemma}

\begin{proof} Let $c$ be an arbitrary non-solution state. Without loss of generality, we can assign colors so that $c(x) = c(y) = 1$, $c(u) = 2$, and $c(v) = 2$ or 3. 

Suppose $c(v) = 2$. Then there are two 1-2 boundary paths between $x$ and $y$. Assume there is no 2-3 or 2-4 internal path between $u$ and $v$. By theorem A.1 (appendix A), there must be both 1-4 and 1-3 internal paths between $x$ and $y$.

Suppose instead that $c(v) = 3$. Then there are 1-2 and 1-3 boundary paths between $x$ and $y$. Assume there is no 2-3 internal path between $u$ and $v$. By theorem A.1, there must be a 1-4 internal path between $x$ and $y$.

In either case, no 1-2 or 1-3 or 1-4 Kempe exchange can result in a solution state. Nor can any Kempe exchange not involving the color 1. Because this is true for every non-solution state of $G_{xy}$, it follows that no sequence of Kempe exchanges starting with $c$ can lead to a solution state. Thus, there is no $n$-$s$ class in the partition of $C_{G_{xy}}$, and by definition 5.1, $G_{xy} \in A^*$.

Conversely, assume that $G_{xy} \in A^*$. Again, let $c$ be an arbitrary non-solution state with $c(x) = c(y) = 1$, $c(u) = 2$, and $c(v) = 2$ or 3.

Suppose $c(v) = 2$. Then, by lemma 5.1, there must be 1-3 and 1-4 internal paths between $x$ and $y$, and by theorem A.1, there can be no 2-4 or 2-3 internal path between $u$ and $v$.

Suppose instead that $c(v) = 3$. Then by lemma 5.1, there must be a 1-4 internal path between $x$ and $y$, and by theorem A.1, there can be no 2-3 internal path between $u$ and $v$.
\end{proof}


\section{The search for members of \textit{A*}}

\noindent Though we are interested in finding any member of $A^*$, we are particularly interested in finding one that is a potential minimal a-graph counterexample. The applicable parent triangulation of a minimal a-graph counterexample must be internally 6-connected; thus, among other attributes, that parent must have minimum degree 5. Because we need the terminology for later use, we prove the well-known result that such a triangulation must have order at least 12. 

\begin{lemma} A minimum degree $5$ planar triangulation has order at least $12$.
\end{lemma}

\begin{proof} Let $v$, $e$, and $f$ denote the number of vertices, edges, and faces, respectively, in a planar triangulation $T$. Let $d_v$ denote the sum over all vertices in $T$ of the degrees of those vertices. Then $d_v = 2e = 3f$. Using Euler's formula, $v - e + f = 2$, we derive $d_v = 6v - 12$. Because $T$ has minimum degree 5, $d_v \geq 5v$. It follows that $v \geq 12$.
\end{proof}

By tedious enumeration of all possible cases, it is straightforward to show, using lemma 5.3, that any $G_{xy} \in A^*$ must have order at least 12 and that the only member of $A^*$ of order 12 is the one depicted in the top panel of figure 6.1{---}it is the graph we identified in section 2 as $G_{xy}^{*}$. The bottom panel shows the two equivalence classes in the state transformation graph for $G^{*}$ whether $G^{*}$ is considered to be $G_{xy}^{*}$ or $G_{uv}^{*}$. Each coloring state has been labeled using a pair of arithmetic signs, the first applying to the coloring of the $(x,y)$ pair of opposite boundary vertices and the second applying to the coloring of the $(u,v)$ pair. An $=$ sign means that the vertices in the pair are colored the same; a $\neq$ sign means that the vertices in the pair are not colored the same. This labeling indicates that for the parented a-graph $G_{xy}^{*}$ there is an $n$ class with 6 states and an $s$ class with 12 states. Thus, $G_{xy}^{*} \in A^{*}$. However, even without finding solution states, we knew that $G_{xy}^{*}$ could not be a minimal a-graph counterexample because its applicable parent triangulation, $G_{xy}^{*} + xy$, is not internally 6-connected: when $x$ and $y$ are joined, $u$ and $v$ have degree 4. The other parent triangulation, $G_{uv}^{*} + uv$, is internally 6-connected{---}it is the 5-regular icosahedron. For the parented a-graph $G_{uv}$, the arithmetic signs in the second position apply; we see that there are non-solution states and solution states in each component. Both equivalence classes are $n$-$s$. Thus, $G_{uv}^{*} \notin A^{*}$. Clearly $G_{uv}^{*}$ is not a minimal a-graph counterexample either.  The a-graph $G^{*}$ is extraordinary:  $G_{xy}^{*} \in A^{*}$, but its applicable parent is not internally 6-connected; $G_{uv}^{*} \notin A^{*}$, but its applicable parent is internally 6-connected.  A \textit{minimal a-graph counterexample} must have an internally 6-connected parent, and, with that parent, belong to $A^{*}${---}that is the type of a-graph we are looking for. However, we shall see that $G^{*}$ is likely as close as we can get.

\begin{figure}[htb]
\label{figure6:1}
\centering
\includegraphics[scale=0.75, trim= 115 300 100 50]{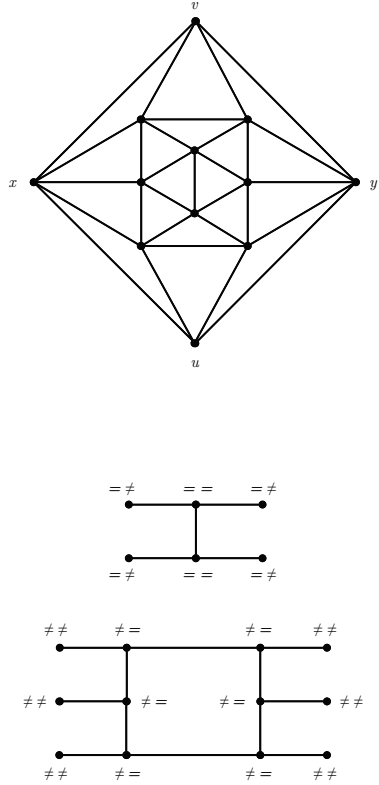}
\caption{In the top panel is the parented a-graph $G_{xy}^{*}$, the only member of $A^*$ of order $12$; in the bottom panel are the two components of its state transformation graph.}
\end{figure}

\begin{remark} Figure 6.1 discloses an interesting fact about the various colorings of the icosahedron. In the bottom panel, if we eliminate all vertices that are labeled with an $=$ sign in the second position, we are left with ten vertices of degree 0{---}this is the state transformation graph for the icosahedron. There are ten distinct colorings, none of which is obtainable from any other by means of Kempe exchanges. However, when the $uv$ edge is deleted from the icosahedron, the state transformation graph for $G_{uv}^{*}$ shows that non-solution states in the component of order 6 connect four of the colorings of the icosahedron. Similarly, non-solution states in the component of order 12 connect six of the colorings of the icosahedron. Starting from any non-solution state of $G_{uv}^{*}$, we can navigate to a solution state and thus to a proper 4-coloring of the icosahedron. 
\end{remark}

In what sense does the ``smallness'' of an a-graph give it a chance to be a member of $A^*$? For $G_{xy}^{*}$, there are only ten \textit{internal} paths between $u$ and $v$ that use no more than a single edge of any triangle and therefore can be 2-colored and maintain a proper coloring of the overall a-graph. For $G_{xy}^{*}$ as drawn in figure 6.1, those ten paths come in mirror-image pairs under reflections in the $u$-$v$ axis. Thus, there are effectively only five distinct paths that can possibly be 2-colored. If we take an uncolored $G_{xy}^{*}$, assign $x$ and $y$ color 1, then select any of those five paths and give it a 2-coloring not involving the color 1, we find that the vertices already colored force the coloring of most, if not all, of the remaining uncolored vertices. Any attempt to color the remaining vertices requires at least one uncolored vertex to be assigned a fifth color. Thus, there is a relatively small number of \textit{coloring degrees of freedom} for $G_{xy}^{*}$ and that leads to forced color assignments at many vertices and inevitably to the need for a fifth color. The same factors will apply to any member of $A^{*}$. Our search for fundamental members of $A^{*}$ will therefore focus on small graphs, which, for purposes of this study, we take to be all a-graphs of order not exceeding 20.

\bigskip

\begin{systematicsearch} Because we are looking particularly for minimal a-graph counterexamples, we will confine our exploration to a-graphs $G$ that are internally 6a-connected. To organize the search efficiently, it is useful to have a taxonomy for such a-graphs. The general structure is depicted in figure 6.2.
\end{systematicsearch}

\begin{figure}[htb]
\label{figure6:2}
\centering
\includegraphics*[scale=0.70, trim= 165 515 150 65]{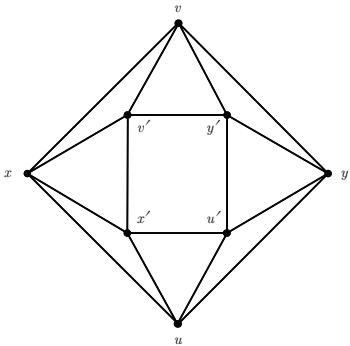}
\caption{Schematic depiction of an a-graph $G$ created by deleting an edge in an arbitrary internally $6$-connected planar triangulation. Vertices within the sides of the interior square $u^{\prime}x^{\prime}v^{\prime}y^{\prime}$ and vertices and edges interior to that square have not been displayed.}
\end{figure}

\noindent This depiction derives from drawing the paths including vertices adjacent to each of the boundary vertices $u$, $x$, $v$, and $y$, and noting that the paths for two adjacent boundary vertices must have a vertex in common because $G$ is a near-triangulation. The four such shared vertices have been designated $u^{\prime}$, $x^{\prime}$, $v^{\prime}$, and $y^{\prime}$ in figure 6.2. Not shown are any vertices inserted into the \textit{sides} of the \textit{ central square} $u^{\prime}x^{\prime}v^{\prime}y^{\prime}$ or any vertices and edges enclosed by that square. We refer to the boundary 4-cycle $uxvy$ as the \textit{outer ring} of $G$ and the cycle that includes $u^{\prime}$, $x^{\prime}$, $v^{\prime}$, and $y^{\prime}$ and all vertices inserted into any or all of the central square's four sides as the \textit{inner ring} of $G$. Even though \textit{inner ring} is a properly precise term, we will often find it convenient to refer more loosely to $u^{\prime}$, $x^{\prime}$, $v^{\prime}$, and $y^{\prime}$ as the \textit{corner vertices} of the square and to the paths between $u^{\prime}$ and $x^{\prime}$, between $x^{\prime}$ and $v^{\prime}$, between $v^{\prime}$ and $y^{\prime}$, and between $y^{\prime}$ and $u^{\prime}$, in each case adjacent to a boundary vertex, as the four \textit{sides} of the square. Any vertices in the sides that are not corner vertices are said to be \textit{side vertices}. We say that $G$ belongs to the family $G(m,n)$ if its inner ring is of order $n$ and encloses $m$ interior vertices. The family $R(m,n)$ derives from taking all graphs in $G(m,n)$ and eliminating their outer rings of order 4 and all edges incident to those outer rings. Up to an isomorphism, a particular a-graph $G$ is uniquely specified by selecting from the applicable family $R(m,n)$ a particular arrangement of corner, side, and interior vertices, and the edges that join them. Any near-triangulation $R \in R(m,n)$ is referred to as a \textit{configuration}.

\begin{definition} A configuration $R \in R(m,n)$ is said to be \textit{conforming} if the following conditions hold:

\begin{itemize}
\item[i.] the corner vertices $u^{\prime}$, $x^{\prime}$, $v^{\prime}$, and $y^{\prime}$ of $R$ have minimum degree $3$,
\item[ii.] there are at least two side vertices in $R$, one each in a pair of opposite sides,
\item[iii.] the $n - 4$ side vertices of $R$ have minimum degree $4$,
\item[iv.] the $m$ vertices interior to the boundary ring of $R$ have minimum degree $5$,
\item[v.] there is no separating triangle or quadrilateral in $R$,
\item[vi.] the only separating pentagons in $R$, if removed, would leave $R$ disconnected into two components, at least one of which has but a single vertex,
\item[vii.] every vertex in the boundary ring of $R$ is adjacent to two and only two other vertices in the boundary ring.
\end{itemize}
\end{definition}

Conditions i through vi assure that the a-graph arising from a conforming configuration is internally 6a-connected. Condition vii is implied by the other conditions but has been included to aid in constructing conforming configurations. If there were an edge between vertices in the boundary ring of $R$ that would otherwise be separated by a distance of at least 2, there would be a violation of one or more of conditions i through vi.

The goal of the systematic search is to find all conforming $R$ up to a given order and then determine which satisfy the coloring condition of lemma 5.3.

\begin{lemma} The $m + n$ interior vertices of $G \in G(m,n)$ have average degree equal to $5 + (m - 2) / (m + n)$.
\end{lemma}

\begin{proof} We use the result developed in the proof of lemma 6.1 that $d_v = 6v - 12$ for a planar triangulation $T$ of order $v$. Let a-graph $G$ be formed by deleting an edge in $T$. Using $d_v$ to represent the sum over all vertices in $G$ instead of $T$ eliminates a count of $2$ from $d_v$ and transforms the equation to $d_v = 6v - 14$, now applying to $G$ instead of $T$. We distinguish between boundary and interior vertices of $G$ by introducing $d_b$ and $d_i$ to represent the sum of the degrees over those two respective sets of vertices. Then $d_v = d_b + d_i$. Also, $v = m + n + 4$. Substituting these into the previous result, we obtain $d_b + d_i = 6(m + n) + 10$. Because $d_b$ counts two boundary edges for each boundary vertex and counts each of the four corner vertices on the inner ring twice, we have $d_b = n + 12$. Thus, $d_i = 5(m + n) + (m - 2)$. Dividing this by $m + n$, the number of interior vertices, yields the desired result.
\end{proof}

\begin{lemma} If $G \in G(0,n)$ or $G(1,n)$, then $G$ is not internally 6a-connected.
\end{lemma}

\begin{proof} When $m = 0$ or $1$, lemma 6.2 shows that the average degree of the interior vertices is less than $5$, implying that at least one interior vertex has degree less than $5$.
\end{proof}

\begin{lemma} If $G \in G(2,n)$, then $G$ is not internally 6a-connected unless $G$ is isomorphic to $G^{*}$.
\end{lemma}

\begin{proof} Let $G \in G(2,n)$. Lemma 6.2 indicates that the interior vertices of $G$ have average degree 5. If $G$ is internally 6a-connected, its interior vertices must have minimum degree 5. Thus, all interior vertices of $G$ have degree 5. We build $G$ with this property by constructing its conforming configuration $R$. By condition ii of definition 6.1, we have $n \geq 6$. Without loss of generality, let $p$ be the vertex in the side from $x^{\prime}$ to $v^{\prime}$ that is adjacent to $v^{\prime}$ and $q$ be the vertex in the side from $u^{\prime}$ to $y^{\prime}$ that is adjacent to $y^{\prime}$. Let $r$ and $s$ be the two vertices interior to the boundary ring of $R$. By conditions iii and vii, $p$ and $q$ must each be joined to both $r$ and $s$, and $r$ and $s$ must be joined to each other. To satisfy condition i without violating condition vii, $v^{\prime}$ and $y^{\prime}$ must each be joined to one of $r$ and $s$ and $x^{\prime}$ and $u^{\prime}$ to the other. The addition of any further side vertices will violate condition iii. With this construction, we note that condition iv is satisfied for both $r$ and $s$, and that conditions v and vi are satisfied for $R$. Thus, we arrive at the conforming configuration $R$ belonging to $R(2,6)$ that is isomorphic to the one depicted in figure 6.1 as a subgraph of $G^{*}$.
\end{proof}

\bigskip

\begin{constructingconfigurations} Because we are looking for $G$ that are internally 6a-connected, definition 6.1 and lemmas 6.3 and 6.4 mean that we need consider only conforming configurations $R(m,n)$ for which $m \geq 3$ and $n \geq 6$, and because we are limiting our search to $G$ of order at most 20, we must construct all such configurations for which $m + n \leq 16$. A key operation in the construction process is a \textit{diamond switch} in which the edge joining one pair of opposite boundary vertices in a $K_{1,1,2}$ diamond is switched to joining the other pair of opposite boundary vertices. The complete process involves the following steps.
\end{constructingconfigurations}

\begin{itemize}
\item[1.] Set $m = 3$ and $n = 6$. Fix the four corner vertices $x^{\prime}$, $v^{\prime}$, $y^{\prime}$, and $u^{\prime}$ for the ring $R$ of a conforming configuration.
\item[2.] Select a $4$-partition $\{n_1,n_2,n_3,n_4\}$ of $n - 4$ that has not already been processed for the given $m$. If all partitions have been processed, skip to step 5. Place $n_1$, $n_2$, $n_3$, and $n_4$ side vertices in the four sides $x^{\prime}$ to $v^{\prime}$, $v^{\prime}$ to $y^{\prime}$, $y^{\prime}$ to $u^{\prime}$, and $u^{\prime}$ to $x^{\prime}$, respectively, accepting the ring only if there is a pair of opposite sides that have at least one vertex each. Reject the ring if it is isomorphic to one already processed and keep repeating this step 2, as necessary, until a distinctly new ring has been created.
\item[3.] Create one or more \textit{representative} conforming configurations by inserting, within the ring established in step 2, $m$ vertices arranged as a path or tree or cactus or polygon or kite or any of those enclosed by a polygon of order at least $6$ or as a single vertex enclosed by polygon of order at least $5$. If this is impossible to accomplish, return to step 2.
\item[4.] For the given $m$, $n$, and partition of $n - 4$, create all conforming configurations by taking the representative configurations created in step 3 and performing all possible sequences of diamond switches. Whenever a sequence of switches ends with a configuration that is either nonconforming or isomorphic to a conforming configuration that has been previously created, discard it and continue until no new non-isomorphic conforming configuration can be created. Return to step 2.
\item[5.] Increase $m$ by $1$. If $m + n \leq 16$, return to step 2.
\item[6.] Increase $n$ by $1$. Set $m = 3$. If $m + n \leq 16$, return to step 2.
\end{itemize}

\noindent Appendix B displays the representative configurations used in the search.

\bigskip

\begin{coloringconfigurations} Every conforming configuration needs to be tested to determine if the corresponding a-graph with boundary vertices $u$, $x$, $v$, and $y$ in place belongs to $A^{*}$. Lemma 5.3 provides the criterion for membership. Depending on the partition of side vertices in the conforming configuration, one parent triangulation can be internally 6-connected while the other is not. Regardless, the corresponding a-graph is internally 6a-connected and we test both $G_{xy}$ (for the presence or absence of an internal 2-color path between $u$ and $v$) and $G_{uv}$ (for the presence or absence of an internal 2-color path between $x$ and $y$). Often we can find a single coloring in which both $x$ and $y$ are colored 1 and both $u$ and $v$ are colored 2 and there are both a 1-3 path between $x$ and $y$ and a 2-3 path between $u$ and $v$ or both a 1-4 path between $x$ and $y$ and a 2-4 path between $u$ and $v$. In such a situation, the single coloring serves to prove that $G_{xy} \notin A^{*}$ and $G_{uv} \notin A^{*}$. That is the case, for example, for the particular $R(6,7)$ conforming configuration shown in figure 6.3. We found that in about half the cases tested, we had to generate two distinct colorings of $G$, one applying to each parent.
\end{coloringconfigurations}

\begin{figure}[htb]
\label{figure6:3}
\centering
\includegraphics*[scale=0.70, trim= 60 555 50 80]{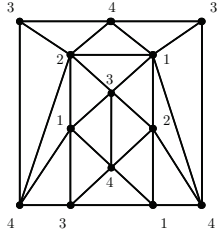}
\caption{A coloring of an $R(6,7)$ conforming configuration that arises from the Errera triangulation of order 17 (see \citep{WeissteinErrera}). The colors of $x$ and $y$ (the left and right boundary vertices, respectively, not shown) are both 1 and the colors of $u$ and $v$ (the bottom and top boundary vertices, respectively, not shown) are both 2.}
\end{figure}

\bigskip

\begin{results} The systematic search turned up no new members of $A^*$. The coloring condition of lemma 5.3 is not satisfied by any conforming configuration other than $R(2,6)$. There are, in fact, other members of $A^*${---}for example, non-fundamental members of orders 19 and 20. Indeed, there are an infinite number of members. But not one was detected in our search because we imposed condition vii of definition 6.1. Among other things, it forbids an edge between vertices in a pair of opposite sides of the central square. If that limitation is removed, we find that the phenomenon occurring for $G_{xy}^{*}$ can be replicated as many times as we like by creating various parented a-graphs, each including a different vertical string of Birkhoff diamonds with common endpoints $x$ and $y$. The non-fundamental member of order 19 has two Birkhoff diamonds in such a string, as does the non-fundamental member of order 20. (The difference between those two non-fundamental members lies in how their two Birkhoff diamond subgraphs are connected.) Because none of the applicable parent triangulations with $x$ and $y$ joined is internally 6-connected, none of these particular non-fundamental members of $A^{*}$ is a minimal a-graph counterexample.
\end{results}

\begin{remark} The Birkhoff diamond is a configuration famous in the history of the 4-color problem. See \citep{Wilson}. It is the subgraph of $G^{*}$ in figure 6.1 induced by all vertices other than $u$ and $v$.
\end{remark}

There is a heuristic argument (but not a proof) that underscores how unlikely it is that a minimal a-graph counterexample exists. Let the planar triangulation $T$ be a minimal counterexample to the $4$-color conjecture. From the proof of lemma 6.1, we know that $e = 3v - 6$, where $e$ and $v$ are the number of edges and vertices, respectively, in $T$. From \citep{Stromquist} we know that $v \geq 52$. So, from theorems 4.1 and 5.1, we see that $T$ must be parent to at least 150 minimal a-graph counterexamples and all must belong to $A^*$. Because our search among internally 6a-connected a-graphs of order at most 20 turned up only the graph $G_{xy}^{*}$ of figure 6.1, it would indeed be stunning to suddenly happen upon a cache of at least $150$. Of course, not all such minimal a-graph counterexamples would necessarily be structurally different from one another. They could, and would likely, belong to a number of isomorphism classes less than $150$. Still, regardless of how symmetrical $T$ may be (when depicted in a manner to highlight its symmetries), it cannot be regular (because there is only one $r$-regular planar triangulation, the icosahedron, for $r = 5$, and none for $r > 5$) and thus has to give rise to several, perhaps many, structurally distinct minimal a-graph counterexamples of high order, a particularly implausible result given our previous heuristic argument that the order of any fundamental member of $A^{*}$ is likely to be low. This line of reasoning leads to the conclusion that if there were to exist but a single structurally distinct minimal a-graph counterexample, the icosahedron would have to be its applicable parent.


\section{Conclusion}

\noindent Proving the 4-color conjecture entails showing that a minimal counterexample does not exist. The standard approach is to generate an unavoidable finite set of reducible configurations.  Unlike the standard approach, we state a \textit{coloring condition} that a minimal counterexample must satisfy. Our approach, which admittedly does not qualify as an alternative proof, also differs in other ways:

\begin{itemize}
\item[1.] Instead of considering a planar triangulation, we study the a-graph created by deleting an edge in that parent triangulation.
\item[2.] Instead of considering the 4-color problem directly, we analyze an equivalent coloring problem, the a-graph coloring problem.
\item[3.] Instead of creating unavoidable reducible configurations, we search for members of $A^*$, the family of a-graphs that satisfy the coloring condition.
\end{itemize}

\noindent The only fundamental member of $A^*$ that our systematic search discovered is the parented a-graph $G_{xy}^{*}$ that contains a single copy of the Birkhoff diamond as a subgraph. Non-fundamental members contain multiple copies of the Birkhoff diamond with common endpoints. None of the parented a-graphs we have discovered that belong to $A^{*}$ satisfies the connectivity condition to be a minimal a-graph counterexample: namely, that the applicable parent is internally 6-connected. Our research strongly suggests that there is no parented a-graph that  satisfies both the coloring condition and the connectivity condition. The two conditions appear to be incompatible. That is the likely reason that no minimal counterexample exists. A byproduct of our approach is the reaffirmation that Kempe chains play a critical role in solving the 4-color problem, despite the entanglement conundrum that Kempe's failed proof encountered (see \citep{Wilson}).


\section*{Acknowledgment}

\noindent I would like to thank Fred Holroyd for his insights and comments regarding virtually all aspects of this article. I would also like to thank a referee for offering a succinct proof of theorem A.1.



\bibliography{agraphreflist}

\newpage

\begin{appendices}

\section{}

\begin{theorem} Let $G$ be an a-graph with boundary cycle $uxvy$ for the exterior $4$-face and let $G$ have a $4$-coloring $c$. Suppose, without loss of generality, that $c(x) =1$, $c(y) = 1$ or $2$, $c(u) = 3$, and $c(v) = 3$ or $4$. Then there is either a $1$-$2$ path between $x$ and $y$ or a $3$-$4$ path between $u$ and $v$.

\begin{proof} Suppose that $G$ with 4-coloring $c$ is a minimal counterexample to the theorem. Clearly $G$ cannot have either an interior $xy$ edge or an interior $uv$ edge. Let $X$ be the set of vertices of $G$ adjacent to $x$; they form an internal path between $u$ and $v$ that includes at least one interior vertex of $G$. At least one of those interior vertices of $G$ belonging to $X$ is colored 2, for otherwise the path between $u$ and $v$ would be colored 3-4, contradicting the supposition. Contract the various edges joining $x$ to each vertex of $X$ colored 2, and change the color of $x$ to 2. The result is a 4-colored a-graph $F$. The edge contractions do not create any 3-4 path between $u$ and $v$. By the minimality assumption, $F$ must therefore have a 1-2 path between $x$ and $y$. Reverse the contractions and restore the color of $x$ to 1 to reveal a 1-2 path between $x$ and $y$ in $G$, contradicting the supposition and establishing the truth of the theorem.
\end{proof}

\end{theorem}


\section{}

\noindent This appendix lists the $106$ representative conforming configurations used to generate all internally 6a-connected a-graphs of order $20$ or lower. The first, $R(2,6)$, leads to $G_{xy}^{*}$, the only fundamental member of $A^*$ that we have discovered. The configurations are listed by increasing order and, within a given order, by increasing number of interior vertices. For a given $m$ and $n$, representative conforming configurations $R(m,n)$ for all non-isomorphic partitions of side vertices among the four sides of the ring are listed. Conforming configurations for some partitions do not exist{---}for example, for $R(3,8)$, $R(4,10)$, and $R(5,11)$, only one partition of side vertices is shown because there are no conforming configurations for any other partition. In a few instances, two or three representatives are listed for a given partition{---}for example, $R(4,8)$. In that particular case, there is a sequence of seven diamond switches that transforms one representative configuration into the other with no intermediate stage that is conforming. To reduce the number of complicated sequences of switches required to generate all possible a-graphs, different representative arrangements of interior vertices for the same partition of side vertices have been included. Another example is $R(6,7)${---}each of the two representatives for one of the partitions can be diamond-switched into figure 6.3: two switches for one of the representatives and only one switch for the other. For some representatives there is a large number of different sequences of diamond switches that lead to non-isomorphic conforming configurations with the same partition of side vertices.

\begin{figure}[htb]
\centering
\includegraphics*[scale=0.85, trim= 125 75 125 75]{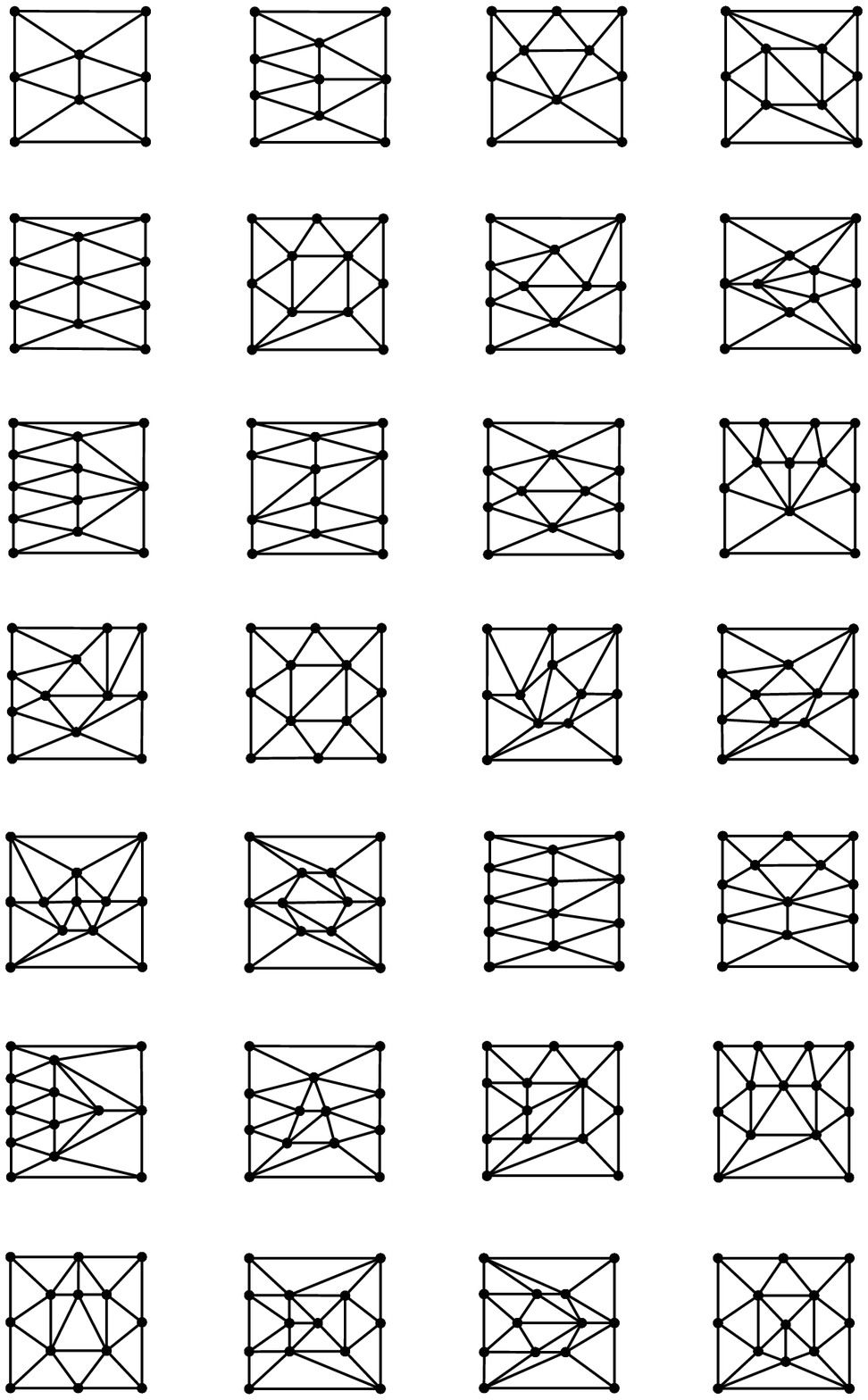}
\end{figure}

\begin{figure}
\centering
\includegraphics*[scale=0.85, trim= 125 75 125 75]{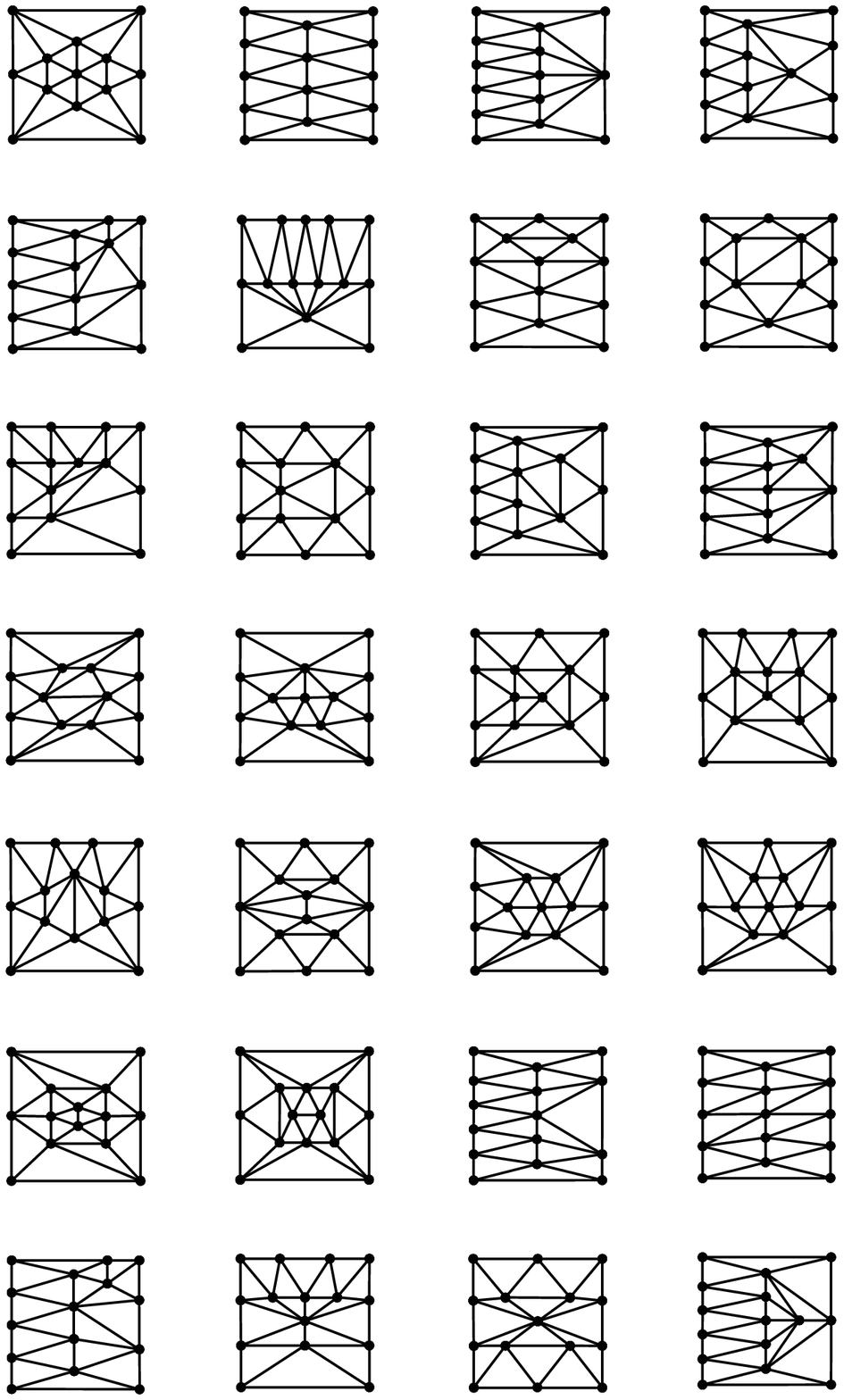}
\end{figure}

\begin{figure}
\centering
\includegraphics*[scale=0.85, trim= 125 75 125 75]{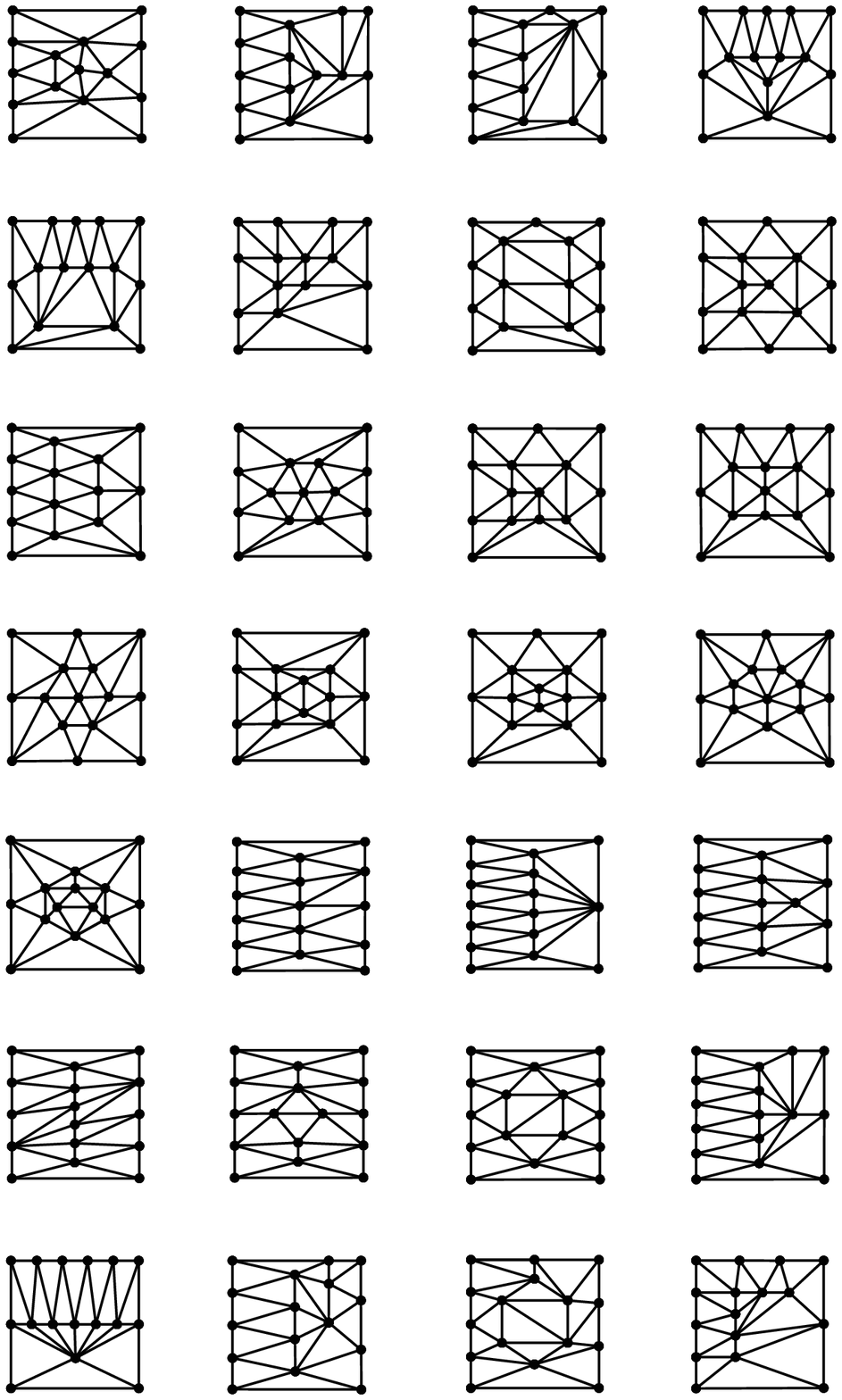}
\end{figure}

\begin{figure}
\centering
\includegraphics*[scale=0.85, trim= 125 75 125 75]{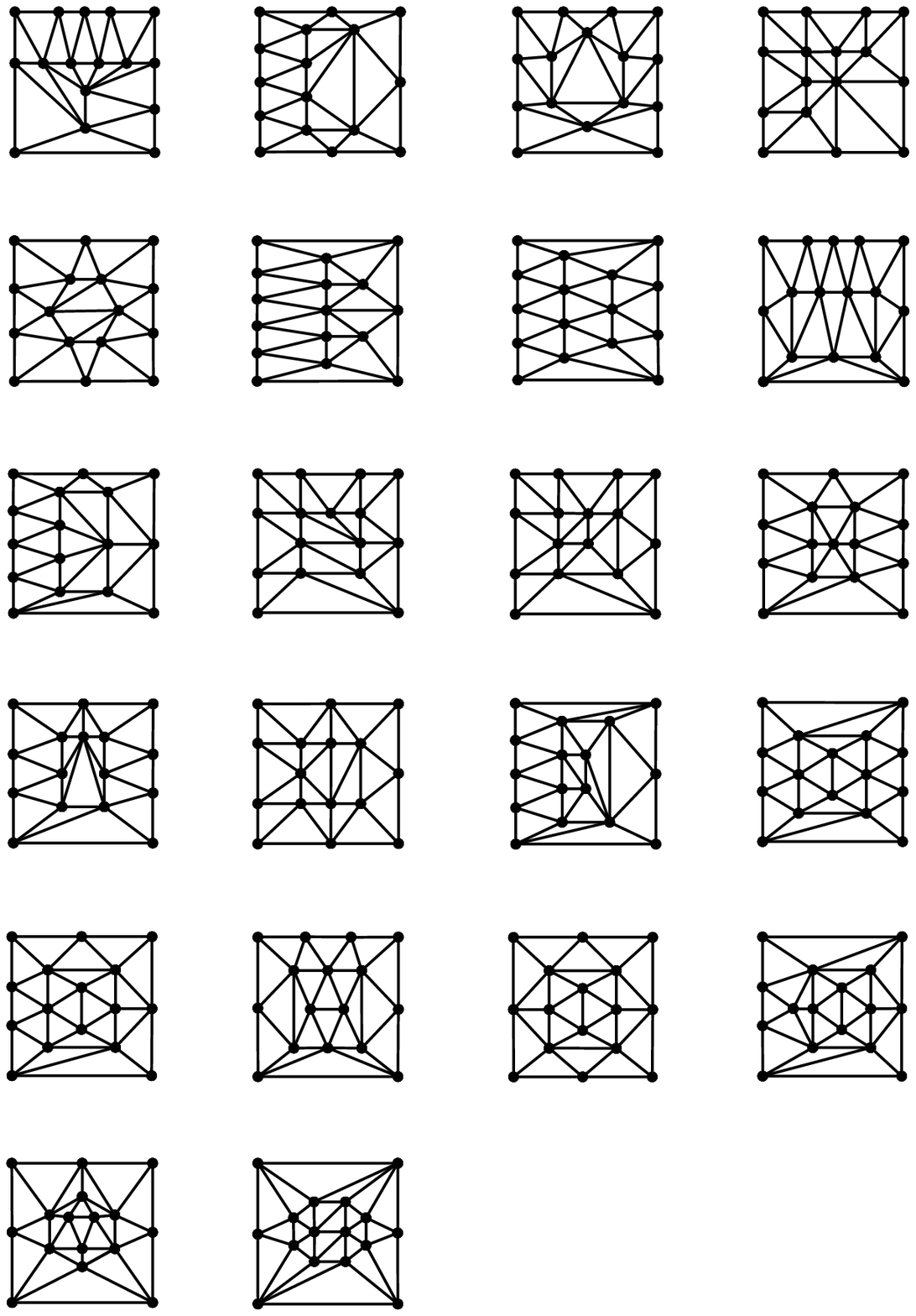}
\end{figure}

\end{appendices}

\end{document}